\documentclass[12pt]{article}

\usepackage{amsmath}
\allowdisplaybreaks[4]
\usepackage{tikz}
\usetikzlibrary{decorations.pathreplacing, matrix}
\usetikzlibrary{decorations.pathreplacing}
\usepackage[all]{xy}
\usepackage{amssymb}
\usepackage{amsthm}
\usepackage{hyperref}
\usepackage{bm}
\hypersetup{colorlinks=true,linkcolor=blue,citecolor=red}
\usepackage{amsmath}
\usepackage{amscd}
\usepackage{verbatim}
\usepackage{eurosym}
\usepackage{float}
\usepackage{color}
\usepackage{dcolumn}
\usepackage[mathscr]{eucal}
\usepackage[all]{xy}
\usepackage{hyperref}
\usepackage{mathrsfs}
\usepackage{amsmath}
\usepackage{amssymb}
\usepackage{amsfonts,ifpdf}
\usepackage{graphicx}
\usepackage{times}
\usepackage{float}
\usepackage{epstopdf}
\usepackage{cite}
\usepackage{youngtab}
\usepackage{ytableau}
\ytableausetup
{mathmode, boxsize=0.9em}
\usepackage{indentfirst}

\newtheorem{theorem}{Theorem}[section]
\newtheorem{lemma}[theorem]{Lemma}
\newtheorem{prop}[theorem]{Proposition}
\newtheorem{coro}[theorem]{Corollary}

\theoremstyle{definition}
\newtheorem{defi}[theorem]{Definition}

\newtheorem{pf}{proof}[section]
\theoremstyle{remark}
\newtheorem{exam}[theorem]{Example}
\newtheorem{remark}[theorem]{Remark}

\makeatletter \@addtoreset{equation}{section} \makeatother
\makeindex 

\def\rank{\mbox{\rm rank}}

\def\rk{\operatorname{rank}}

\def\A{\mathcal{A}}

\def\B{\mathcal{B}}
\renewcommand\ell{l}

\def\det{\mbox{\rm det}}
\def\And{\mbox{\rm ~and~}}

\def\sst{\scriptscriptstyle}

\begin{document}

\begin{center}
{\Large\bf
$k$-Adjoint of Hyperplane Arrangements
}\\ [7pt]
\end{center}

\vskip 3mm

\begin{center}
 Weikang Liang$^{1}$, Suijie Wang$^{2}$ and Chengdong Zhao$^{3}$\\[8pt]
 $^{1,2}$School of Mathematics\\
 Hunan University\\
 Changsha 410082, Hunan, P. R. China\\[12pt]

 $^{3}$
 School of Mathematics and Statistics\\
 Central South University\\
 Changsha 410083, Hunan, P. R. China\\[15pt]

 Emails: $^{1}$kangkang@hnu.edu.cn, $^{2}$wangsuijie@hnu.edu.cn,  $^{3}$cdzhao@csu.edu.cn\\[15pt]

\end{center}

\vskip 3mm

\begin{abstract}

In this paper, we introduce the $k$-adjoint of a given hyperplane arrangement $\mathcal{A}$ associated with rank-$k$ elements in the intersection lattice $L(\mathcal{A})$, which generalizes
the classical adjoint proposed by Bixby and Coullard.
The $k$-adjoint of $\mathcal{A}$ induces a decomposition of the Grassmannian,
which we call the $\mathcal{A}$-adjoint decomposition.
Inspired by the work of Gelfand, Goresky, MacPherson, and Serganova,
we generalize the matroid decomposition and refined Schubert decomposition of the Grassmannian from the perspective of $\mathcal{A}$.
Furthermore, we prove that these three decompositions are exactly the same decomposition.
A notable application involves providing a combinatorial classification of all the $k$-dimensional restrictions of $\mathcal{A}$.
Consequently, we establish the  anti-monotonicity property of some combinatorial invariants, such as Whitney numbers of the first kind and the independece numbers.

\vskip 6pt

\\ [7pt]
\end{abstract}

\section{Introduction}
In this paper, we introduce an approach for constructing new hyperplane arrangements using the intersection lattice of a given hyperplane arrangement, which is referred to as the $k$-adjoint arrangement.
We reveal the connections and applications of $k$-adjoint arrangements to the following objects:
\begin{itemize}
\item Extending the concept of the adjoint of a hyperplane arrangement, as proposed by Bixby and Coullard \cite{Bixby-Coullard1988};

\item Providing a new decomposition of the Grassmannian, while also extending two classical decompositions attributed to Gelfand, Goresky, MacPherson, and Serganova \cite{Gelfand-Goresky-MacPherson-Serganova1987}. Moreover, we demonstrate that the three decompositions are the same;

\item Classifying all $k$-dimensional restrictions of a given hyperplane arrangement;

\item Establishing the anti-monotonicity property of certain combinatorial invariants of a hyperplane arrangement when restricted to different subspaces \cite{Chen-Fu-Wang-2021}.
\end{itemize}

Of utmost initial motivation to us is a fundamental question: how to classify the combinatorial structure of a given hyperplane arrangement restricted to different subspaces.
To rigorously formulate our question, we recall the basic notation in hyperplane arrangements.
A {\it hyperplane arrangement} $\mathcal{A}$ is a finite collection of hyperplanes in the vector space $\mathbb{F}^n$ over a field $\mathbb{F}$.
Throughout this paper,
we assume that $\mathcal{A}$ is {\it linear} and  {\it essential},
i.e.,
$\bigcap_{H \in \mathcal{A}} H=\{0\}$.
Let $L(\mathcal{A})$ be the {\it intersection lattice} of $\mathcal{A}$
consisting of subspaces $X=\bigcap_{H \in \mathcal{B}} H$ for all
$\mathcal{B} \subseteq\mathcal{A}$,
equipped with a partial order induced from the reverse inclusion,
where the minimal element is set to be $\mathbb{F}^n=\bigcap_{H \in  \varnothing} H$.
It is known that $L(\A)$ is graded with the rank function $\rk(X)=n-\dim(X)$.
Let
\[L_{k}(\mathcal{A}):=\left\{X \in L(\mathcal{A})  \mid \rk(X)=k\right\}.\]
Given a subspace $U$ of $\mathbb{F}^n$,
the {\it restriction} of $\mathcal{A}$ to $U$ is the arrangement in $U$ defined by
\[\mathcal{A}|_{U}=\{H \cap U \mid  H \in \mathcal{A}\And U\not\subseteq H\}.\]
If $\dim(U) = k$,
$\mathcal{A}|_{U}$ is called a $k$-dimensional restriction of $\A$.
The initial question can be rephrased as classifying all $k$-dimensional restrictions of $\A$,
or in other words,
characterizing $L(\mathcal{A}|_{U})$ for all $k$-dimensional subspaces $U$.

We address the question by introducing a new hyperplane arrangement associated with $L_k(\mathcal{A})$,
called the $k$-adjoint of $\A$.
We primarily note that our question is equivalent to finding a geometric decomposition of the Grassmannian $\mathrm{Gr}(k, n, \mathbb{F})$ ($\mathrm{Gr}(k, n)$ for simplicity),
which stands for the set of all $k$-dimensional subspaces of $ \mathbb{F}^n$.
Let $[n]=\{1,2,\ldots,n\}$ and let
$\mathbb{F}^{[n]\choose k}$ be the ${n\choose k}$-dimensional vector space over $\mathbb{F}$ with columns indexed by $k$-subsets of $[n]$.
Then $U\in \mathrm{Gr}(k, n)$ can be specified as the row space of a $k\times n$ matrix $A$, known as a matrix representative of $U$.
Then up to a nonzero scalar, the vector $(\Delta_{I}(U))_{I\in{\binom{[n]}{k}}}\in \mathbb{F}^{[n]\choose k}$ is independent of the choice of matrix representatives $A
$ of $U$, denoted by $\Delta(U)$, where $\Delta_{I}(U)$ denotes the $k\times k $ minor of $A$ with the column index set $I$.
In 1988,
Bixby and Coullard \cite{Bixby-Coullard1988} noticed that $L_{n-1}(\mathcal{A})$ yields an adjoint of the hyperplane arrangement $\A$.
As an extension, our definition of $k$-adjoint of $\A$ is as follows.

\begin{defi}\label{k-adjoint-def}
For each $X\in L_{k}(\mathcal{A})$ and $I\in {[n]\choose k}$$(0 \leq k < n)$,
let
\[a_{I}(X)=(-1)^{\frac{1}{2}k(k+1) + \sum_{i\in I} i}\Delta_{[n]\setminus I}(X).\]
Then the \emph{adjoint} $H(X)$ of $X$ is a hyperplane in $\mathbb{F}^{{[n]}\choose{k}}$ defined by
\begin{align*}
H(X)=\Big\{({x}_I)_{\sst I\in{\binom{[n]}{k}}}\;\big|\;\sum_{ I\in{\binom{[n]}{k}}}
a_I(X)\cdot {x}_I=0\Big\}.
\end{align*}
The hyperplane arrangement
\[
\mathcal{A}^{(k)}=\{H(X) \mid X\in L_{k}(\mathcal{A})\}
\]
in $\mathbb{F}^{{[n]}\choose{k}}$ is called \emph{$k$-adjoint} of $\mathcal{A}$.
\end{defi}

The concept of $k$-adjoint plays a central role in constructing a new geometric decomposition of the Grassmannian,
as demonstrated below.
Given any hyperplane arrangement $\mathcal{A}$ in $\mathbb{F}^n$, for each $X \in L(\mathcal{A})$, the {\it relative interior} of $X$ is defined as  
\begin{align}\label{relint-def}
    X^{\circ} = X \setminus \bigcup_{Y \in L(\mathcal{A}), Y > X} Y.
\end{align}  
With this notation, for each $ P \in L(\mathcal{A}^{(k)}) $, we define  
\[
\mathcal{S}_P = \left\{ U \in \mathrm{Gr}(k, n) \mid \Delta(U) \in P^\circ \right\}.
\]
As will be seen in Section \ref{Sec-3},
the Grassmannian can be decomposed as
\[\mathrm{Gr}(k, n)=\bigsqcup_{P\in L(\mathcal{A}^{(k)})}\mathcal{S}_P,\]
called \emph{$\A$-adjoint decomposition} of $\mathrm{Gr}(k, n)$,
and each $\mathcal{S}_P$ is called an \emph{$\A$-adjoint stratum}.

Gelfand, Goresky, MacPherson and Serganova \cite{Gelfand-Goresky-MacPherson-Serganova1987}
provided three equivalent methods of decomposing the Grassmannian.
Inspired by their work, we propose two additional decompositions of the Grassmannian associated with a given hyperplane arrangement $\A$,
the  {\it $\A$-matroid decomposition} and the
{\it refined $\A$-Schubert decomposition},
which are actually the same as the $\A$-adjoint decomposition.
Roughly speaking, the \(\mathcal{A}\)-matroid decomposition decomposes
\(\mathrm{Gr}(k, n)\) into strata such that the \(k\)-dimensional restriction of \(\mathcal{A}\) to every subspace within a given stratum has the same matroid structure.
The refined $\A$-Schubert decomposition is the common refinement of all the decompositions into Schubert cells induced from the maximal chains of $L(\A)$.
When $\A$ is the Boolean arrangement,
they both reduce to the two decompositions in \cite{Gelfand-Goresky-MacPherson-Serganova1987}.
Detailed definitions for these decompositions can be found in Section \ref{Sec-3}, and the fact that they are the same is stated in the following theorem.

\begin{theorem}\label{main}
For a given hyperplane arrangement $\A$ in $\mathbb{R}^n$, the $\A$-matroid decomposition, the $\A$-adjoint decomposition and the refined $\A$-Schubert decomposition are exactly the same decomposition of the Grassmannian.
\end{theorem}

At this point Theorem \ref{main} answers the question raised at the beginning of the paper,
namely the intersection lattice of the $k$-dimensional restriction $\mathcal{A}|_{U}$ remains isomorphic to the flat lattice of a fixed matroid when $U$ varies over a fixed $\A$-adjoint stratum.
When $P$ runs over the lattice $L(\mathcal{A}^{(k)})$,
we are naturally led to explore the relationships on combinatorial invariants of $\A|_U$ for $U\in \mathcal{S}_P$.
Let $\mathfrak{M}(\mathcal{A})$ be the matroid associated with the arrangement $\A$. 
We denote by $I_i(\A)$ the independence number of size $i$ in the matroid 
$\mathfrak{M}(\mathcal{A})$, and by $w_i(\A)$ the $i$-th Whitney number of the first kind of $\mathfrak{M}(\mathcal{A})$. 
Then, 
we find that signless Whitney numbers $|{w}_i(\mathcal{A}|_{U})|$ of the first kind and the independence numbers $I_i(\mathcal{A}|_{U})$ exhibit the anti-monotonicity property, as stated below.

\begin{theorem}\label{wf}
Let $\mathcal{A}$ be a hyperplane arrangement in $\mathbb{R}^n$
and $\mathcal{A}^{(k)}$ be the $k$-adjoint of $\mathcal{A}$.
Assume that $P_1, P_2\in L(\mathcal{A}^{(k)})$ with $P_1\leq P_2$.
If $U_j\in \mathcal{S}_{P_j}$ for $j=1,2$,
then we have
\begin{enumerate}
\item[{\rm (1)}] $I_i(\mathcal{A}|_{U_1})\geq I_i(\mathcal{A}|_{U_2});$
\item[{\rm (2)}] $|{w}_i(\mathcal{A}|_{U_1})|\geq |w_i(\mathcal{A}|_{U_2})|$.
\end{enumerate}
\end{theorem}

\section{$k$-Adjoint of Hyperplane Arrangement}\label{Sec-2}

In this section,
the basic properties of the $k$-adjoint are preliminarily explored, along with the motivation behind our construction.
We also calculate the $k$-adjoint of the product of two hyperplane arrangements. As a consequence, we obtain the $k$-adjoint of the Boolean arrangement.

Given a hyperplane arrangement $\A$ in $\mathbb{F}^n$,
let $L_{k}(\mathcal{A})=\left\{X \in L(\mathcal{A})  \mid \rk(X)=k\right\}$.
Recall Definition \ref{k-adjoint-def} that for each subspace $X \in L_k(\A)$, its $k$-adjoint $H(X)$ is a hyperplane in $\mathbb{F}^{{[n]}\choose{k}}$ given by 
\begin{align}\label{k-adj}
	H(X)=\Big\{(x_I)_{I\in\binom{[n]}{k}} \,\big|\,\sum_{I\in\binom{[n]}{k}}a_I(X)\cdot x_I = 0\Big\},
\end{align}
where the coeﬃcients
\[
a_I(X)=(-1)^{\frac{1}{2}k(k+1)+\sum_{i\in I}i}\cdot\Delta_{[n]\setminus I}(X).
\]
Then the collection of all these $H(X)$ forms the $k$-adjoint $\A^{(k)}$ of $\A$.

\begin{remark}
	The concept of $k$-adjoint is a generalization of adjoint arrangement due to Bixby and Coullard \cite{Bixby-Coullard1988}.
	For $k=n-1$, all members of $L_{n-1}(\A)$ are one-dimensional subspaces,
	denoted by
	$L_{n-1}(\mathcal{A})=\{\mathbb{F}{\bm u}_1,\mathbb{F}{\bm u}_2,\ldots,\mathbb{F}{\bm u}_{s}\}$,
	where each $\mathbb{F}{\bm u}_i$ is the subspace spanned by ${\bm u}_i$.
	Let ${\bm u}_i=(u_{i,1}, u_{i,2}, \ldots, u_{i,n})^T$.
	By invoking \eqref{k-adj}, we get
	\begin{align}\label{adjoint-arrangement-k}
		H(\mathbb{F}{\bm u}_i)
		=\Big\{({x}_1, {x}_2, \cdots, {x}_n) \in \mathbb{F}^n \,\big|\,\sum_{j=1}^{n}(-1)^{j}u_{i,j}{x}_j=0\Big\}.
	\end{align}
	Implicitly stated by Bixby and Coullard,
	the adjoint arrangement
	${\cal K} =
	\{{K}_{{\bm u}_1},{K}_{{\bm u}_2},\ldots,{K}_{{\bm u}_s}\}$ is a linear arrangement in $\mathbb{F}^n$,  where ${K}_{{\bm u}_i}$ is defined by
	\begin{align}\label{adjoint-arrangement-0}
		{K}_{{\bm u}_i}
		=\Big\{({x}_1, {x}_2, \cdots, {x}_n) \in \mathbb{F}^n\,\big|\, \sum_{j=1}^{n}u_{i,j}{x}_j=0\Big\}.
	\end{align}
	By comparing \eqref{adjoint-arrangement-k} and \eqref{adjoint-arrangement-0},
	we see that
	$L(\A^{(n-1)})$ is indeed isomorphic to $L(\mathcal{K})$,
	which implies our $k$-adjoint generalizes Bixby and Coullard's adjoint.
	
	Additionally, it is worth noting that there are three trivial special cases to be considered:
	\begin{itemize}
		\item The $0$-adjoint of $\A$ is the origin in $\mathbb{F}$;
		\item The $1$-adjoint of $\A$ is itself;
		\item The $n$-adjoint of  $\A$ is the empty arrangement in $\mathbb{F}$.
	\end{itemize}
\end{remark}

The following lemma characterizes a relationship between a given $k$-dimensional subspace and the elements in $L_k(\A)$,
which is the basis for defining $k$-adjoint.

\begin{lemma}\label{lem-2.1}
	Let $U\in\mathrm{Gr}(k, n)$ and $X \in L_k(\mathcal{A})$ with $0 < k < n$. Then $\mathbb{F}^n=U\oplus X$ if and only if $\Delta(U)\notin H(X)$.
\end{lemma}
\begin{proof}
Let $A_U$ and $A_X$ be matrix representatives of $U$ and $X$ respectively. Note that $X$ is of dimension $n-k$, and hence we define an $n \times n$ matrix by
\[
M_{U,X} = \begin{bmatrix}
	A_U \\
	A_X
\end{bmatrix}.
\]
Then $\mathbb{F}^n = U \oplus X$ if and only if $\det M_{U,X} \neq 0$. Employing Laplace’s expansion theorem on the first $k$ rows, we obtain
\[
\det M_{U, X} = \sum_{I\in\binom{[n]}{k}}(-1)^{\frac{1}{2}k(k + 1)+\sum_{i\in I}i}\cdot\Delta_{[n]\setminus I}(A_{X})\cdot\Delta_I(A_U).
\]
It immediately follows from the definition of $H(X)$ in \eqref{k-adj} that $\mathbb{F}^n=U\oplus X$ if and only if $\Delta(U)\notin H(X)$.
\end{proof}

In the following,
we shall discuss the $k$-adjoint of the product of two hyperplane arrangements.
Denote by $(\A,\mathbb{F}^n)$ the hyperplane arrangement $\A$ in $\mathbb{F}^n$.
The product $(\mathcal{A} \times \mathcal{B},\mathbb{F}^{n+m})$ of two hyperplane arrangements $(\A,\mathbb{F}^n)$ and $(\B,\mathbb{F}^m)$ is defined by
$$
\mathcal{A} \times \mathcal{B}=\left\{H \oplus \mathbb{F}^m \mid H \in \mathcal{A}\right\} \cup\left\{\mathbb{F}^n \oplus K \mid K \in \mathcal{B}\right\} .
$$
A hyperplane arrangement is called {\it reducible} if it can be written as a product of two hyperplane arrangements.
There is a natural isomorphism of lattices
\begin{align}\label{prod-arr}
\pi\colon L\left(\mathcal{A}\right) \times L\left(\mathcal{B}\right) \longrightarrow L\left(\mathcal{A} \times \mathcal{B}\right)
\end{align}
given by the map $\pi\left(Y, Z\right)=Y \oplus Z$, for example
see \cite[Proposition 2.14]{Orlik-Terao}.

\begin{defi}
The tensor product
$\left(\A \otimes \B, \mathbb{F}^n \otimes \mathbb{F}^m\right)$  of two hyperplane arrangements $(\A,\mathbb{F}^n)$ and $(\B,\mathbb{F}^m)$ is defined by
\[
\A \otimes \B=
\left\{H \otimes \mathbb{F}^m+\mathbb{F}^n \otimes K \mid H \in \A \,\,{\text {and}}\,\, K \in \B \right\} .
\]
\end{defi}
\vspace{-\parskip}\noindent
Note that this is well-defined since
$\dim(H \otimes \mathbb{F}^m + \mathbb{F}^n \otimes K) = nm-1$
for $H\in\A$ and $K\in \B$.

\begin{lemma}\label{tensor-arr}
	Let $\{e_i\}_{1\le i \le n}$ and $\{e_j'\}_{1 \le j \le m}$ be bases of $\mathbb{F}^n$ and $\mathbb{F}^m$ respectively.
	Assume that $H \subseteq \mathbb{F}^n$ and $K \subseteq \mathbb{F}^m$ are hyperplanes defined by
	\begin{align*}
	H &= \Big\{\sum_{1\leq i \leq n}{ x_i} e_i \,\big|\, a_1{x}_1+a_2{x}_2+\cdots + a_n{x}_n=0\Big\}\\[6pt]
	K &= \Big\{\sum_{1\leq j \leq m} {y}_j e_j' \,\big|\, b_1{ y_1}+b_2{y_2}+\cdots + b_m{ y_m}=0\Big\}.
	\end{align*}
	Then the hyperplane $H \otimes \mathbb{F}^m+\mathbb{F}^n \otimes K$ of
	$\,\mathbb{F}^n \otimes \mathbb{F}^m$ is given by
	\[
	H \otimes \mathbb{F}^m+\mathbb{F}^n \otimes K
	=
	\Big\{ \sum_{1\le i\le n} \sum_{1\le j\le m}
			{z}_{i,j}
			e_i\otimes e_j'\,\,\big|\,
			\sum_{1\le i\le n} \sum_{1\le j\le m}
	{a_ib_j}{z}_{i,j}=0
	\Big\}.\]
\end{lemma}

\begin{proof}
Let
\[P= \Big\{ 
 \sum_{1\le i\le n} \sum_{1\le j\le m}
			{z}_{i,j}
			e_i\otimes e_j'\,\,\big|\,
 \sum_{1\le i\le n} \sum_{1\le j\le m}
	{a_ib_j}{z}_{i,j}=0
\Big\}.
\]
For any $v=\sum_{1\le i\le n} {x}_ie_i\in \mathbb{F}^{n}$
and $w=\sum_{1\le j\le m} {y}_je_j'\in \mathbb{F}^{m}$,
we have
\[v\otimes w=  \sum_{1\le i\le n} \sum_{1\le j\le m}
			{ x_i} { y_j}
			e_i\otimes e_j' \in P\]
			if and only if
	\[
	 \sum_{1\le i\le n} \sum_{1\le j\le m}
	a_ib_j{x}_i{y}_j=
	\Big(\sum_{1\leq i \leq n}a_i{x}_i\Big)\Big(\sum_{1\leq j 	\leq m}b_j{y}_j\Big)=0.
	\]
Thus $v \in H$ or $w \in K$ implies $v \otimes w \in P$.
Then $H \otimes \mathbb{F}^m+\mathbb{F}^n \otimes K \subseteq P$.
Moreover, $H \otimes \mathbb{F}^m+\mathbb{F}^n \otimes K = P$ since they are of the same dimension.
\end{proof}

The following proposition is a decomposition formula of the $k$-adjoint of reducible hyperplane arrangements.

\begin{prop}\label{tensor}
Let $(\A,\mathbb{F}^{n})$ and $(\B,\mathbb{F}^{m})$ be two  hyperplane arrangements.
Then we have
\begin{align*}
(\A \times \B)^{(k)}= \prod_{i= 0}^{k}\A^{(i)} \otimes \B^{(k-i)}.
\end{align*}
In general, we have
\begin{align*}
(\A_1 \times \A_2\times \cdots\times \A_q)^{(k)}= \prod_{ i_1+ i_2+\cdots+ i_q =k }\A_1^{(i_1)} \otimes \A_2^{(i_2)}\otimes
\cdots  \otimes\A_q^{(i_q)}.
\end{align*}
\end{prop}

\begin{proof}
Restricting the natural isomorphism in \eqref{prod-arr} to $L_k(\A \times \mathcal{B})$ ,
we get the bijection
\[\pi^{-1}\colon L_k(\A \times \mathcal{B}) \longrightarrow \bigcup_{i=0}^k{L_i(\A) \times L_{k-i}(\mathcal{B})}.\]
Given $X\in L_k(\A\times \B)$,
let
\[\pi^{-1}(X) = (Y,Z)\in L_\ell(\A)\times L_{k-\ell}(\mathcal{B})\]
for some $ \ell$,
and hence $X=Y\oplus Z$.
Let $A_Y$ and $A_Z$ be matrix representatives of $Y$ and $Z$ respectively.
Then we obtain a matrix representative of $X$ given by
\begin{equation*}
{A_X} =
\begin{tikzpicture}[baseline=(matrix.center)]
  \matrix (matrix) [matrix of math nodes, left delimiter=[, right delimiter={]}] {
    A_{Y} & O \\
    O & A_{Z} \\
  };

  \draw [decorate,decoration={brace,amplitude=3pt,raise=1pt}]
    ([xshift=15pt]matrix-1-2.north east) -- ([xshift=15pt]matrix-1-2.south east) node[midway, right=4pt] {\small $n-\ell$};
  \draw [decorate,decoration={brace,amplitude=3pt,raise=1pt}]
    ([xshift=11pt]matrix-2-2.north east) -- ([xshift=11pt]matrix-2-2.south east) node[midway, right=4pt] {\small $m-(k-\ell)$};

\draw [decorate,decoration={brace,amplitude=5pt,mirror,raise=2pt}]
    ([yshift=-4pt]matrix-2-1.south west) -- ([yshift=-2pt]matrix-2-2.south east) node[midway,below=4pt] {\small $n+m$};
\end{tikzpicture}
\end{equation*}
where $O$ denotes the zero matrix of the corresponding size.
For any $J\in {[n+m] \choose k}$,
write $J=\{j_1,j_2,\ldots,j_{k}\}$ with
$j_1<j_2<\ldots<j_{k}$.
Let $J_1=\{j_1,j_2,\ldots,j_{\ell}\}$ and $J_2=\{j_{\ell+1},j_{i+2},\ldots,
j_{k}\}$.
Then, elementary linear algebra implies that
\begin{align*}
	\Delta_{[n+m]\setminus J}(A_X)=
\begin{cases}
      \Delta_{[n]\setminus J_1}(A_Y)\Delta_{[m]\setminus (J_2-n)}(A_Z), &\text{if } j_\ell\leq n<j_{\ell+1};\\[3pt]
      0, &\text{otherwise}.
    \end{cases}
\end{align*}
Here we use ${J}_2-n$ to denote the set obtained by decreasing all elements in $J_2$ by $n$.
Combining this with the definition of $k$-adjoint arrangement in \eqref{k-adj},
we get
\begin{align*}
H(X) = \Big\{(x_I)_{I\in\binom{[n+m]}{k}} \,\big|\,
(-1)^{k(k-\ell)}\sum_{I_1\in {[n]\choose \ell}}\sum_{I_2\in {[m]\choose k- \ell}}
a_{I_1}(Y)b_{I_2}(Z){x}_{I_1\cup(I_2+n)}=0\Big\},
\end{align*}
where
\[a_{I_1}(Y)= (-1)^{\frac{1}{2}\ell(\ell+1) + \sum_{i_1\in I_1} i_1}
\Delta_{[n]\setminus I_1}(A_Y)\]
and
 \[b_{I_2}(Z)= (-1)^{\frac{1}{2}(k-\ell)(k-\ell+1) + \sum_{i_2\in I_2} i_2}
\Delta_{[m]\setminus I_2}(A_Z).\]
By Lemma \ref{tensor-arr},
we see that
\[H(X) =
\Big(H(Y) \otimes \mathbb{F}^{[m]\choose k-\ell} + \mathbb{F}^{[n]\choose \ell} \otimes H(Z)\Big)\oplus
\Big(\bigoplus_{\substack{0\leq i \leq k\\[2pt] i\neq \ell}}
\mathbb{F}^{[n]\choose k-i}\otimes\mathbb{F}^{[m]\choose i}\Big),\]
which completes the proof.
\end{proof}

We shall provide a concrete example of the $k$-adjoint.
Recall that the Boolean arrangement in $\mathbb{F}^n$ is defined as
\begin{align}\label{Bool-arr}
	B_n=\{ {x}_i = 0\mid i\in [n]\}.
\end{align}

\begin{coro}\label{boolean}
The $k$-adjoint of Boolean arrangement $B_n^{(k)}$ is the Boolean arrangement $B_{n \choose k}$ in
$\mathbb{F}^{[n] \choose k}$.
\end{coro}

\begin{proof}
Note that
\[B_n= \underbrace{ B_1 \times B_1 \times\cdots\times B_1}_{n}.\]
Applying Proposition \ref{tensor},
the result follows from facts that
$B_1^{(0)}=B_1$,  $B_1^{(1)}= \varnothing$ and $B_1\otimes B_1 = B_1$.
\end{proof}

\begin{exam}
We illustrate by taking $n=4$ and $k=2$
and list all the subspaces of rank $2$ in $L(B_4)$ by their matrix representatives as follows:
\begin{align*}
&X_1= \begin{bmatrix}
1 & 0 & 0 & 0\\
0 & 1 & 0 & 0
\end{bmatrix},
\;
X_2= \begin{bmatrix}
1 & 0 & 0 & 0\\
0 & 0 & 1 & 0
\end{bmatrix},
\;
X_3 = \begin{bmatrix}
1 & 0 & 0 & 0\\
0 & 0 & 0 & 1
\end{bmatrix},\\[5pt]
&X_4 =
\begin{bmatrix}
0 & 1 & 0 & 0\\
0 & 0 & 1 & 0
\end{bmatrix},
\;
X_5=\begin{bmatrix}
0 & 1 & 0 & 0\\
0 & 0 & 0 & 1
\end{bmatrix},
\;
X_6=\begin{bmatrix}
0 & 0 & 1 & 0\\
0 & 0 & 0 & 1
\end{bmatrix}.
\end{align*}
By the definition of $k$-adjoint \eqref{k-adj}, we have
\begin{align*}
	&H(X_1) \colon {x}_{34}=0,\quad
	H(X_2) \colon -{x}_{24}=0,\quad
	H(X_3) \colon {x}_{23}=0,\quad \\[5pt]
	&H(X_4) \colon {x}_{14}=0,\quad
	H(X_5) \colon -{x}_{13}=0,\quad
	H(X_6)  \colon {x}_{12}=0.
\end{align*}
We see that $B_4^{(2)}$ is indeed the Boolean arrangement $B_{6}$.
\end{exam}

\section{$\A$-Decompositions of the Grassmannian}\label{Sec-3}

This section focuses on elucidating three decompositions of the real Grassmannian via a given hyperplane arrangement $\A$.
Inspired by the celebrated work of Gelfand, Goresky, MacPherson and Serganova \cite{Gelfand-Goresky-MacPherson-Serganova1987},
we shall introduce the $\A$-matroid decomposition and the refined $\A$-Schubert decomposition.
The primary goal of this section is to show that the $\A$-matroid decomposition, the refined $\A$-Schubert decomposition, and our $\A$-adjoint decomposition are actually the same decomposition of the Grassmannian, as shown in Theorem \ref{main}.
	
\subsection{Three $\A$-Decompositions of the Grassmannian}\label{Sec-3.1}

Throughout this section, we assume that  $$\A = \{H_1, \ldots, H_m\}$$ is a hyperplane arrangement in the Euclidean space $\mathbb{R}^n$,
where each $H_i$ is given by its normal vector $\alpha_i$ for $i = 1, 2, \ldots, m$, 
that is,
\[
H_i = \{v \in \mathbb{R}^n \mid \langle\alpha_i, v\rangle = 0\}
\]
with $\langle \cdot, \cdot\rangle$ denoting the standard inner product in $\mathbb{R}^n$.
We now proceed to provide detailed definitions for the three $\A$-decompositions of the Grassmannian.

First,
let us review the definition of the $\A$-adjoint decomposition as mentioned in the introduction.
It is straightforward to verify that the \( k \)-adjoint \(\mathcal{A}^{(k)}\) induces a disjoint decomposition  
\[
\mathbb{R}^{\binom{[n]}{k}} = \bigsqcup_{P \in L(\mathcal{A}^{(k)})} P^{\circ},
\]
where \( P^{\circ} \) is defined in \eqref{relint-def}.
This yields the following decomposition of \(\mathrm{Gr}(k, n)\).
\begin{defi}\label{A-adjoint}
	For each $P \in L(\mathcal{A}^{(k)})$, let
	\begin{align*}
		\mathcal{S}_P:=\{U\in\mathrm{Gr}(k, n)\mid\Delta(U)\in P^{\circ}\}.
	\end{align*}
	Then we call
\begin{equation}\label{TFG1}
	\mathrm{Gr}(k, n)=\bigsqcup_{\substack{P\in L(\mathcal{A}^{(k)})}}\mathcal{S}_P
\end{equation}
the \emph{$\A$-adjoint decomposition} of the Grassmannian and each $\mathcal{S}_P$ is called an \emph{$\A$-adjoint stratum}.
\end{defi}

Next we will define the $\A$-matroid decomposition of the Grassmannian. 
Refer to \cite{Oxley} for basic concepts on matroids.
For each $U \in \mathrm{Gr}(k, n)$, let $\beta_i = \mathrm{Proj}_U\alpha_i$ denote the orthogonal projection of $\alpha_i$ in $U$, where $\alpha_i$ is the normal vector of $H_i \in \mathcal{A}$ for $i = 1, \ldots, m$. 
We can construct a rank-$k$ matroid $\mathfrak{M}_{\mathcal{A}}(U)$ on $[m]$ by defining its rank function on subsets $I \subseteq [m]$ as $\rank(I) := \dim \mathrm{span}\{\beta_i \mid i \in I\}$.
Then we obtain a subclass
\[
\mathbf{M}(\mathcal{A}):=\{\mathfrak{M}_{\mathcal{A}}(U)\mid U\in\mathrm{Gr}(k, n)\}
\]
of matroids on the ground set $[m]$, associated with the hyperplane arrangement $\A$.

\begin{defi}\label{A-matroid}
For each matroid $\mathfrak{M} \in \mathbf{M}(\mathcal{A})$, let
\[
\Omega_{\mathcal{A}}(\mathfrak{M}) := \left\{ U \in \mathrm{Gr}(k, n) \,\big|\, \mathfrak{M}_{\mathcal{A}}(U) = \mathfrak{M} \right\}.
\]
Then we call
\begin{equation}\label{TFG2}
	\mathrm{Gr}(k, n) = \bigsqcup_{\mathfrak{M} \in \mathbf{M}(\mathcal{A})} \Omega_{\mathcal{A}}(\mathfrak{M})
\end{equation}
the \emph{$\mathcal{A}$-matroid decomposition} of the Grassmannian and each $\Omega_{\mathcal{A}}(\mathfrak{M})$ is called an \emph{$\mathcal{A}$-matroid stratum}.
\end{defi}

From the above definition, given $\mathfrak{M}\in\mathbf{M}(\mathcal{A})$ and $U_1,U_2 \in \Omega_{\A}(\mathfrak{M})$,
$\mathfrak{M}_{\mathcal{A}}(U_1)=\mathfrak{M}_{\mathcal{A}}(U_2) = \mathfrak{M}$
implies that the intersection lattices
$L(\A|_{U_1})$ and $L(\A|_{U_2})$ are isomorphic, see \cite[p. 425 Proposition 3.6]{Stanley}.
Therefore,
the $\A$-matroid decomposition \eqref{TFG2} indeed gives a classification of $L(\mathcal{A}|_{U})$ for all $k$-dimensional subspaces $U$. 
That is, \( L(\mathcal{A}|_{U}) \) remains isomorphic for all \( U \) within a fixed \(\mathcal{A}\)-matroid stratum.
It also should be noted that if we take $\A$ to be the Boolean arrangement $B_n$ given in \eqref{Bool-arr},
the $\A$-matroid decomposition coincides with GGMS's matroid decomposition in \cite[page 12, 1.2 Definition]{Gelfand-Goresky-MacPherson-Serganova1987}.

The last decomposition generalizes the classical Schubert decomposition by employing the hyperplane arrangement $\A$.
Basic notation and concepts on Schubert decomposition can be found in \cite{Fulton-1997}.

Let \(\boldsymbol{F}\colon \{0\} = F_0 \subsetneq F_1 \subsetneq \cdots \subsetneq F_n = \mathbb{R}^n\) be a flag where each \(F_i\) is an \(i\)-dimensional subspace of \(\mathbb{R}^n\), 
and let $\sigma=\{i_1,\ldots,i_k\}\subseteq [n]$ be a $k$-subset with elements ordered increasingly. 
The {\it Schubert cell} associated with \(\boldsymbol{F}\) and \(\sigma\) is defined as   
\begin{align}\label{def:sc}
\Omega_{\boldsymbol{F}}(\sigma):=\{U\in\mathrm{Gr}(k, n)\mid\dim (U\cap F_{i_l}) > \dim (U\cap F_{i_{l}-1}), \;1\leq l\leq k\}.
\end{align}
The {\it Schubert cell decomposition} means the disjoint union
\begin{equation}\label{tfg3}
\mathrm{Gr}(k, n)=\bigsqcup_{\sigma\in\binom{[n]}{k}}\Omega_{\boldsymbol{F}}(\sigma).
\end{equation}
Note that each maximal chain of $L(\mathcal{A})$ is a flag of $\mathbb{R}^n$. 
Let $\mathbf{mc}(L(\mathcal{A}))$ denote the set of maximal chains of $L(\mathcal{A})$. 
It is clear that  
\begin{align*}
\mathrm{Gr}(k, n) 
&=
\bigcap_{\boldsymbol{F} \in \mathbf{mc}(L(\mathcal{A}))}
\bigsqcup_{\sigma(\boldsymbol{F}) \in \binom{[n]}{k}}
\Omega_{\boldsymbol{F}}(\sigma(\boldsymbol{F}))\\[6pt]
& = \bigsqcup_{\sigma\colon \mathbf{mc}(L(\mathcal{A})) \to \binom{[n]}{k}}
\bigcap_{\boldsymbol{F} \in \mathbf{mc}(L(\mathcal{A}))}
\Omega_{\boldsymbol{F}}(\sigma(\boldsymbol{F})).
\end{align*}

\begin{defi}\label{A-Sucbert}
A \emph{Schubert symbol} of \( \mathcal{A} \) is a map \( \sigma\colon\mathbf{mc}(L(\mathcal{A}))\longrightarrow \binom{[n]}{k} \) such that  
\[
\Omega_{\A}(\sigma):=\bigcap_{\boldsymbol{F}\in\mathbf{mc}(L(\mathcal{A}))}\Omega_{\boldsymbol{F}}(\sigma(\boldsymbol{F}))\neq \varnothing.
\]
The \emph{refined $\A$-Schubert decomposition} refers to the disjoint decomposition  
\begin{equation}\label{TFG3}
\mathrm{Gr}(k, n)=\bigsqcup_{\sigma}\Omega_{\A}(\sigma),
\end{equation}  
where \( \sigma \) ranges over the Schubert symbols of \( \mathcal{A} \).
\end{defi}

\subsection{Three $\mathcal{A}$-Decompositions Are the Same}

This section is dedicated to proving Theorem \ref{main}, which establishes that the $\A$-adjoint decomposition, the $\A$-matroid decomposition, and the refined $\A$-Schubert decomposition defined in the previous subsection are the same decomposition of the Grassmannian.

Fixed $U \in \mathrm{Gr}(k, n)$, it follows from Lemma \ref{lem-2.1} that for any $X \in L_k(\A)$, $X$ satisfies $U\oplus X = \mathbb{R}^n$ if and only if $\Delta(U)\notin H(X)$. We introduce the following definition, which plays a key role in the proof of Theorem \ref{main}.

\begin{defi}\label{def:lu}
	For each vector subspace $U\in\mathrm{Gr}(k, n)$ with $0 < k < n$, let
	\begin{align*}
		L_U(\mathcal{A}):&=\{X\in L_k(\mathcal{A})\mid\Delta(U)\notin H(X)\}\\[6pt]
		&=\{X\in L_k(\mathcal{A})\mid U\oplus X=\mathbb{R}^n\};\\[6pt]
		L^U(\mathcal{A}):&=\{X\in L_k(\mathcal{A})\mid \Delta (U)\in H(X)\}\\[6pt]
		&=\{X\in L_k(\mathcal{A})\mid U\cap X\neq \{0\}\}.
	\end{align*}
\end{defi}

\begin{lemma}\label{lem-2}
Let $U$ be a vector subspace of $\mathbb{R}^n$. 
Then for each subset $I$ of $[m]$,
\[
\dim\Big(U\cap\bigcap_{i\in I}H_i\Big)=\dim U-\dim(\mathrm{span}\{\beta_i\mid i\in I\})
\]
and the rank function of the matroid $\mathfrak{M}_\mathcal{A}(U)$ satisfies
\[
\mathrm{rank}(I)=\dim U-\dim\Big(U\cap\bigcap_{i\in I}H_i\Big).
\]
\end{lemma}

\begin{proof}
When $|I| = 0$, that is, $I$ is empty, we have $\mathbb{R}^n=\bigcap_{i\in I}H_i$ and the result clearly holds. Consider the case $I$ with $|I|\geq 1$. Notice that
\[
U\cap\bigcap_{i\in I}H_i=\{v\in U \mid \langle v,\beta_i\rangle = 0\text{ for }i\in I\}.
\]
The dimension formula follows immediately. So does the rank formula.
\end{proof}

\begin{lemma}\label{coro-2}
Let $U\in\mathrm{Gr}(k, n)$ and $I\subseteq [m]$. Then the following statements are equivalent:
\begin{enumerate}
	\item[{\rm (1)}] $I$ is a basis of the matroid $\mathfrak{M}_\mathcal{A}(U)$.
	\item[{\rm (2)}] $|I| = k$ and $\{\beta_i \mid i\in I\}$ is linearly independent.
	\item[{\rm (3)}] $|I| = k$ and $\bigcap_{i\in I}H_i\in L_U(\mathcal{A})$.
\end{enumerate}
\end{lemma}
\begin{proof}
	Note that $k = \dim U$ and $\bigcap_{H \in \mathcal{A}} H=\{0\}$. We claim that $U=\mathrm{span}\{\beta_i \mid i\in [m]\}$. Suppose that $U':=\mathrm{span}\{\beta_i\mid i\in [m]\}$ is properly contained in $U$. Then $\mathrm{span}\{\alpha_i\mid i\in [m]\} \subseteq U'\oplus U^{\perp}$ is properly contained in $\mathbb{R}^n$, which contradicts $\mathbb{R}^n=\mathrm{span}\{\alpha_i\mid i\in [m]\}$. Now it is trivial that (1) is equivalent to (2).
	
	(2) $\Leftrightarrow$ (3) Let $|I| = k$ and $\{\beta_i \mid i \in I\}$ be linearly independent. Then $\{\beta_i \mid i \in I\}$ forms a basis of $U$, $\{\alpha_i \mid i \in I\}$ is linearly independent and $\dim(\bigcap_{i\in I}H_i)=n - k$. For each $v\in U\cap\bigcap_{i\in I}H_i$, let $v = \sum_{i\in I}a_i\beta_i$. We have
	\[
	\langle v,\alpha_j\rangle=\sum_{i\in I}a_i\langle\beta_i,\beta_j\rangle = 0,\quad \ j\in I.
	\]
	Since $\{\beta_i \mid i \in I\}$ is linearly independent, we see that $\det[\langle\beta_i,\beta_j\rangle]\neq 0$. It follows that all $a_i$ with $i \in I$ are zero, so $v = 0$. Since $\dim(U) + \dim(\bigcap_{i\in I}H_i) = n$, we conclude that $U\oplus\bigcap_{i\in I}H_i=\mathbb{R}^n$, and hence $\bigcap_{i\in I}H_i\in L_U(\mathcal{A})$.
	
	Conversely, let $|I| = k$ and $\mathbb{R}^n = U\oplus\bigcap_{i\in I}H_i$. It is trivial that $\alpha_i$ with $i\in I$ are linearly independent. Suppose $\{\beta_i \mid i \in I\}$ is linearly dependent. There exists a nonzero vector $v\in U$ such that $\langle v,\beta_i\rangle = 0$ for all $i\in I$. For each $i \in I$, let $\alpha_i=\beta_i+\gamma_i$, where $\gamma_i\in U^{\perp}$. Then $\langle v,\alpha_i\rangle=\langle v, \beta_i\rangle+\langle v,\gamma_i\rangle = 0$ for all $i\in I$. Thus $v\in U\cap\bigcap_{i\in I}H_i$, contradicting the given direct sum.
\end{proof}

\begin{proof}[Proof of Theorem \ref{main}]
	First, we show that the $\A$-adjoint decomposition \eqref{TFG1} is the same as the $\A$-matroid decomposition \eqref{TFG2}. 
Assume that $\mathcal{S}_P$ is a stratum of the $\mathcal{A}$-adjoint decomposition \eqref{TFG1}, where $P\in L(\mathcal{A}^{(k)})$.
 Let $U\in \mathcal{S}_P$. 
 It means that $\Delta(U)\in P^{\circ}$, which is equivalent to
	\[
	\Delta(U)\in P\quad\text{and} \quad \Delta(U)\notin Q
	\]
	for all $Q\in L(\mathcal{A}^{(k)})$ such that $Q > P$.
Invoking Definition \ref{def:lu},
we obtain
	\begin{align*}
	\bigcap_{X\in L^{U}(\mathcal{A})} H(X)&= P =\bigcap_{X\in L(\mathcal{A}),P\subseteq H(X)} H(X)
    \end{align*}
and
    \begin{align*}
	P\setminus\bigcup_{X\in L_U(\mathcal{A})} H(X)&= P^{\circ} =P\setminus\bigcup_{X\in L(\mathcal{A}),P\nsubseteq H(X)} H(X).
	\end{align*}
	It follows that 
\begin{align}
L^{U}(\mathcal{A}) = \{X \in L_k(\mathcal{A}) \mid P \subseteq H(X)\}
\end{align}	
and
\begin{align}\label{eq:def-pf}
 L_U(\mathcal{A}) = \{X \in L_k(\mathcal{A}) \mid P \nsubseteq H(X)\}.
\end{align}	
Thus for $U_1,U_2\in\mathcal{S}_P$, we have $L^{U_1}(\mathcal{A}) = L^{U_2}(\mathcal{A})$ and $L_{U_1}(\mathcal{A}) = L_{U_2}(\mathcal{A})$. Applying Lemma \ref{coro-2}, we obtain $\mathfrak{M}_\mathcal{A}(U_1)= \mathfrak{M}_\mathcal{A}(U_2)$. Let $\mathfrak{M}\in\mathbf{M}(\mathcal{A})$ denote the common matroid $\mathfrak{M}_\mathcal{A}(U)$ for all $U\in\mathcal{S}_P$. We see that $\mathcal{S}_P\subseteq\Omega_{\A}(\mathfrak{M})$.
	
Conversely, let $\Omega_{\A}(\mathfrak{M})$ be a stratum of the $\mathcal{A}$-matroid decomposition \eqref{TFG2}. 
For $U_1,U_2\in\Omega_{\A}(\mathfrak{M})$, that is, $\mathfrak{M}=\mathfrak{M}_{\A}(U_1)=\mathfrak{M}_{\A}(U_2)$. 
Lemma \ref{coro-2} implies that $L_{U_1}(\mathcal{A}) = L_{U_2}(\mathcal{A})$; consequently, $L^{U_1}(\mathcal{A}) = L^{U_2}(\mathcal{A})$. Let 
	\[
	P = \bigcap_{X \in L^{U_1}(\A)}H(X) =  \bigcap_{X \in L^{U_2}(\A)}H(X) \in L(\A^{(k)}).
	\]
	Then $\Delta(U_1)$ and $\Delta(U_2)$ both lie in $P^\circ$. 
	We obtain that all members of $\Omega_{\A}( \mathfrak{M})$ are in the same stratum $\mathcal{S}_P$ of the $\A$-adjoint decomposition \eqref{TFG1}, which implies $\Omega_{\A}( \mathfrak{M}) \subseteq \mathcal{S}_P$.
	
Next we show that the refined $\A$-Schubert decomposition \eqref{TFG3} coincides with the $\A$-matroid decomposition \eqref{TFG2}.
Assume that $\Omega_\A(\sigma)$ is a stratum of the decomposition \eqref{TFG3}, where $\sigma$ is a Schubert symbol of $\A$. 
Let $U \in \Omega_\A(\sigma)$. 
For each flag $\boldsymbol{F} \in\mathbf{mc}(L(\mathcal{A}))$,
we write
\(\boldsymbol{F}\colon \{0\} = F_0 \subsetneq F_1 \subsetneq \cdots \subsetneq F_n = \mathbb{R}^n \)
and
$\sigma(\boldsymbol{F}) = \{i_1(\boldsymbol{F}), \ldots, i_k(\boldsymbol{F})\}$.
Then we have 
\[
\dim \left(U\cap F_{i_j(\boldsymbol{F})}\right) > \dim \left(U\cap F_{i_{j}(\boldsymbol{F})-1}\right)
\]
for any $j \in \{1, 2,\ldots, k\}$ and $\dim (U\cap F_{i_{1}(\boldsymbol{F})-1}) = 0$.  
Therefore, for each $X \in L_k(\A)$, 
we have $U\oplus X=\mathbb{R}^n$ if and only if there exists some $\boldsymbol{F}\in\mathbf{mc}(L(\mathcal{A}))$ with $\sigma(\boldsymbol{F}) = \{n-k+1, n-k+2, \ldots, n\}$ such that $X = F_{n-k}$. 
Again, by Definition \ref{def:lu}, we have
\[
L_U(\A) = \{F_{n-k} \mid \boldsymbol{F}\in\mathbf{mc}(L(\mathcal{A}))\,\And\, \sigma(\boldsymbol{F}) = \{n-k+1, \ldots, n\}\}.
\]
Thus for $U_1,U_2\in \Omega_\A(\sigma)$, 
we have $L_{U_1}(\mathcal{A}) = L_{U_2}(\mathcal{A})$. 
Applying Lemma \ref{coro-2}, we obtain $\mathfrak{M}_{\A}(U_1) = \mathfrak{M}_{\A}(U_2)$. Let $\mathfrak{M}\in\mathbf{M}(\mathcal{A})$ denote the common matroid $\mathfrak{M}_\A(U)$ for all $U\in\Omega_\A(\sigma)$. We see that $\Omega_\A(\sigma)\subseteq\Omega_{\A}(\mathfrak{M})$.

Conversely, assume that $\Omega_{\A}(\mathfrak{M})$ is a stratum of the $\A$-matroid decomposition \eqref{TFG2}. Let $U \in \Omega_{\A}(\mathfrak{M})$. For each subset $I \subseteq [m]$, Lemma \ref{lem-2} implies that
\[
\dim\Big(U\cap\bigcap_{i\in I}H_i\Big) = k-\mathrm{rank}(I).
\]
Therefore, for each maximal chain \(\boldsymbol{F}\colon \{0\} = F_0 \subsetneq F_1 \subsetneq \cdots \subsetneq F_n = \mathbb{R}^n\), 
the dimension $\dim(U \cap F_i)$ is determined by the matroid $\mathfrak{M}$ for $i = 1, \ldots, n$. This implies that the unique Schubert symbol $\sigma$ of $\A$ such that $U \in \Omega_\A(\sigma)$ is determined by the matroid $\mathfrak{M}$. We obtain that all members of $\Omega_{\A}(\mathfrak{M})$ are in the same stratum $\Omega_{\A}(\sigma)$ of the decomposition \eqref{TFG3}, that is, $\Omega_{\A}(\mathfrak{M}) \subseteq \Omega_{\A}(\sigma)$. This completes the proof.
\end{proof}

\section{Combinatorial invariants}\label{Sec-4}

In this section, we will present a proof of Theorem \ref{wf}, which establishes the anti-monotonicity property of the independence numbers and Whitney numbers of the first kind as $U$ varies  across different $\A$-adjoint strata.

We place our problem within the framework of matroid theory.
Let $\mathfrak{M}$ be a matroid on $[m]$. 
Any subset of a basis is called an {\it independent set}, while a {\it dependent set} is a subset that is not independent. 
A  {\it circuit} is a minimal dependent set, and a  {\it flat} is a subset of $[m]$ whose rank increase when adding any other element. It is known that the collection of all flats of $\mathfrak{M}$, ordered by inclusion,
forms a lattice $L(\mathfrak{M})$ with a unique minimal element $\hat{0}$, i.e.,
the intersection of all flats.
An independent set of $\mathfrak{M}$ is called a {\it broken circuit}
if it is obtained from a circuit of $\mathfrak{M}$ by removing its maximal element under a given total order.
The {\em characteristic polynomial} $\chi_\mathfrak{M}(t)$ of $\mathfrak{M}$ is
\[\chi_\mathfrak{M}(t)=\sum_{x \in L(\mathfrak{M})} \mu(\hat{0}, x) t^{k-\rk(x)}
=\sum_{i=0}^{k}w_it^{k-i},\]
where $k=\rk(\mathfrak{M})$ and $\mu$ is the M\"obius function of $L(\mathfrak{M})$.
Each coefficient ${ w }_i(\mathfrak{M})$ is called the
{\em $i$-th Whitney number of the first kind}.
Whitney's celebrated NBC (no-broken-circuit) theorem \cite{Whitney1932} gives a combinatorial interpretation on the Whitney number of the first kind.
\begin{theorem}[NBC Theorem \cite{Stanley}]\label{NBCthm}
Let $\mathfrak{M}$ be a matroid on $[m]$.
Then
$|w_i(\mathfrak{M})|$ is the number of independent $i$-sets of $\mathfrak{M}$ containing no broken circuit for $i=0,1,\ldots,k$.
\end{theorem}

Given a hyperplane arrangement $\mathcal{A}=\{H_1,H_2,\ldots,H_m\}$ in $\mathbb{R}^n$, the set of all normal vectors of hyperplanes in $\A$ defines a matroid $\mathfrak{M}(\mathcal{A})$ on $[m]$.
Note that $L(\mathfrak{M}(\mathcal{A}))$ is isomorphic to $L(\mathcal{A})$. 
Hence, the related concepts in the hyperplane arrangement $\A$, such as independence numbers and characteristic polynomial, are naturally inherited from $\mathfrak{M}(\mathcal{A})$.
Now we are prepared to prove Theorem \ref{wf}.

\begin{proof}[Proof of Theorem \ref{wf}]
We claim that $L_{U_2}(\mathcal{A}) \subseteq  L_{U_1}(\mathcal{A})$.
By the same argument in the proof of Theorem \ref{main}, we obtain
\begin{align}\label{cr-1}
L_{U_j}(\mathcal{A}) = \{ X \in L_k(\A) \mid P_j \not\subseteq H(X) \}
\end{align}
for \( j = 1, 2 \); see \eqref{eq:def-pf}.
For any $X\in L_{U_2}(\mathcal{A})$,
we have
$P_2\not\subseteq H(X)$
by \eqref{cr-1}.
It follows from $P_2 \subseteq P_1$ that
$P_1\not\subseteq H(X)$.
So we have $X\in L_{U_1}(\mathcal{A})$ as claimed.

To prove that $I_i(\mathcal{A}|_{U_1})\geq I_i(\mathcal{A}|_{U_2})$,
it suffices to show that each basis $J$ of $\mathfrak{M}_{\A}(U_2)$
is a basis of $\mathfrak{M}_{\A}(U_1)$.
It follows from Lemma \ref{coro-2} that $\bigcap_{j\in J} H_j \in L_{U_2}(\mathcal{A}) \subseteq L_{U_1}(\mathcal{A})$. Again by Lemma \ref{coro-2}, $J$ also serves as a basis for $\mathfrak{M}_{\A}(U_1)$, as desired. 

To prove that $|{w}_i(\mathcal{A}|_{U_1})|\geq |w_i(\mathcal{A}|_{U_2})|$, note that each dependent set of $\mathfrak{M}_{\A}(U_1)$ is dependent in $\mathfrak{M}_{\A}(U_2)$.
It follows that each broken circuit of $\mathfrak{M}_{\A}(U_1)$ has a subset which is a broken circuit of $\mathfrak{M}_{\A}(U_2)$. Therefore,
any subset of $[m]$ containing no broken circuit of $\mathfrak{M}_{\A}(U_2)$ does not contain broken circuit of $\mathfrak{M}_{\A}(U_1)$.
By Whitney's NBC (no-broken-circuit) Theorem \ref{NBCthm},
we have $|{w}_i(\mathcal{A}|_{U_1})|\geq |w_i(\mathcal{A}|_{U_2})|$.
\end{proof}

\noindent
\textbf{\large Acknowledgments.}
This work was done under the auspices of the National Science Foundation of China (12101613).
We sincerely thank the anonymous reviewers for their insightful comments and valuable suggestions, which have greatly improved this manuscript.

\end{document}